\documentclass{amsart}
\usepackage{latexsym,amsthm,amssymb}

\usepackage{bbm}

\DeclareMathAlphabet{\mathpzc}{OT1}{pzc}{m}{it}

\usepackage[leqno]{amsmath}

\usepackage{mathrsfs}

\usepackage[all]{xy}

\usepackage{color}
\usepackage{colordvi}
\usepackage[dvips]{graphicx}           

\usepackage{verbatim}

\usepackage{ifthen}
\newboolean{hide.proofs}
\setboolean{hide.proofs}{false}

\newtheorem{theorem}{Theorem}[section]
\newtheorem{lemma}[theorem]{Lemma}

\newtheorem{proposition}[theorem]{Proposition}

\newtheorem{remark}[theorem]{Remark}
\newtheorem{noname}[theorem]{}

\newtheorem*{acknowledgement}{Acknowledgement}

\newtheorem{lemma-conjecture}[theorem]{Lemma--Conjecture}

\newtheorem{corollary}[theorem]{Corollary}

\newtheorem{example}[theorem]{Example}

\numberwithin{equation}{theorem}

\renewcommand{\mathcal}{\mathscr}

\newcommand{\SE}{{\mathcal{E}}}

\newcommand{\SO}{{\mathcal{O}}}

\renewcommand{\mathbb}{\mathbf}

\title{On the deformations of canonical double covers of minimal rational surfaces}

\author{Francisco Javier Gallego}
\author{Miguel Gonz\'alez}
\author{\\ Bangere P. Purnaprajna}

\address{Departamento de \'Algebra, Universidad Complutense de Madrid}
\email{gallego@mat.ucm.es}
\address{Departamento de \'Algebra, Universidad Complutense de Madrid}
\email{mgonza@mat.ucm.es}
\address{Department of Mathematics, University of Kansas}
\email{purna@math.ku.edu}

\begin{document}

\begin{abstract}
The purpose of
this article is to study the deformations of smooth surfaces $X$ of general type whose canonical map is a finite, degree $2$ morphism onto a minimal rational surface or onto $\mathbf F_1$, embedded in projective space by a very ample complete linear series.  Among other things, we prove that any deformation of the canonical morphism of such surfaces $X$ is again a morphism of degree $2$.  A priori, this is not at all obvious, for the invariants $(p_g(X),c_1^2(X))$ of most of these surfaces lie on or above the Castelnuovo line; thus, in principle, a deformation of such $X$ could have a birational canonical map. We also map the region of the geography of surfaces of general type corresponding to the invariants of the surfaces $X$ and we compute the dimension of the irreducible moduli component containing $[X]$. In certain cases we exhibit some interesting moduli components parametrizing surfaces $S$ whose canonical map has different behavior but whose invariants are the same as the invariants of $X$. One of the interests of the article is that we prove the results about moduli spaces employing crucially techniques on deformation of morphisms. The key point or our arguments is the use of a criterion that requires only infinitesimal, cohomological information of the canonical morphism of $X$. As a by--product, we also prove the non--existence of ``canonically'' embedded multiple structures on minimal rational surfaces and on $\mathbf F_1$.

\end{abstract}

\maketitle

\section*{Introduction}

In
this article we study the deformations of smooth surfaces $X$ of general type whose canonical map is a finite, degree $2$ morphism onto a minimal rational surface $Y$, embedded in projective space by a very ample complete linear series.  Minimal rational surfaces are the projective plane and Hirzebruch surfaces $\mathbf F_e$ with $e \neq 1$; then, for completeness' sake we will include canonical double covers of all Hirzebruch surfaces (also the non--minimal $\mathbf F_1$) in our study.

\medskip
Our main result is Theorem~\ref{coversofP2andruled}, in which we show that any deformation of the canonical morphism of most such surfaces $X$ is again a morphism of degree $2$. Thus in Section~\ref{non.existence.onminimal.section} we study how the deformations of the canonical morphism $\varphi$ of $X$ are. Since $\varphi$ is by hypothesis finite and of degree $2$ onto its image, any deformation of $\varphi$ has to be also a finite morphism and either of degree $2$ or of degree $1$. A priori, it is not clear which of the two possibilities should happen. On the one hand, it is well--known (see~\cite{Hor2p_g-4}) that if $\varphi$ is a double cover of a surface of minimal degree, any deformation of $\varphi$ has again degree $2$. Of course, there is a good reason for this to occur, namely, the invariants of a canonical double cover of a surface of minimal degree are below Castelnuovo's line $c_1^2=3p_g-7$ (recall that no surface below Castelnuovo's line can have a birational canonical map). On the other hand however, if $\, Y$ is embedded by an arbitrary very ample complete linear series, there are many surfaces $X$ whose invariants lie on or above Castelnuovo's line. Thus the existence of deformations of $\varphi$ to degree $1$ morphisms would be plausible in these cases (for instance, Examples~\ref{example39} and~\ref{example45} show the existence of surfaces of general type with very ample canonical divisors and having the same invariants as certain canonical double covers of minimal rational surfaces). Moreover, if we consider canonical double covers of non--minimal rational surfaces, there exist cases for which the canonical morphism can be deformed to a morphism of degree $1$ (see~\cite[4.5]{AK} and~\cite[Theorem 3.14]{MP}). In this article we completely settle this matter when $Y$ is $\mathbf P^2$ or a Hirzebruch surface and the branch divisor of the canonical double cover is base--point--free. In Theorem~\ref{coversofP2andruled}
we show that any deformation of $\varphi$  is again a degree $2$ morphism. Theorem~\ref{coversofP2andruled} can be rephrased in terms of the moduli space in the following way: the irreducible component containing $[X]$ it its moduli space parametrizes surfaces whose canonical map is a finite, degree $2$ morphism. Theorem~\ref{coversofP2andruled} contrasts therefore with the results
of~\cite[4.5]{AK} and~\cite[Theorem 3.14]{MP} mentioned above;  instead, Theorem~\ref{coversofP2andruled} generalizes
the behavior of the deformations of canonical covers of surfaces of minimal degree mentioned earlier.    
This variety of behaviors
shows how different and
complicated the moduli spaces of surfaces of general type that are canonical double covers of arbitrarily embedded rational surfaces can be when compared with either the
moduli of curves or, for that matter, the moduli of surfaces of general type that are canonical double covers of surfaces of minimal degree.

\medskip

The key point for the proof of Theorem~\ref{coversofP2andruled} is the vanishing
of Hom$(\mathcal I/\mathcal I^2,\omega_Y(-1))$, which we prove in Proposition~\ref{nonexist.canonical.structures}.
This vanishing can be interpreted as the non--existence, at an ``infinitesimal'' level, of deformations
of $\varphi$ to degree $1$ morphisms.
Thus Theorem~\ref{coversofP2andruled} shows how infinitesimal properties translate into global
statements such as the fact that the moduli component
containing
$[X]$ parametrizes surfaces whose canonical map is a degree $2$ morphism.
On the other hand, finite morphisms that can be deformed to morphisms of degree $1$ are linked
to the existence of certain multiple structures embedded in projective
space (see~\cite{Fong}, \cite{GPcarpets}, \cite{Gon} and~\cite{Compositio}).
Thus, as a by--product of Proposition~\ref{nonexist.canonical.structures}, in Section~\ref{noropes}
we prove the non--existence of ``canonically'' embedded
(i.e., such that the dualizing sheaf is the restriction of $\mathcal O_{\mathbf P^N}(1)$)
multiple structures on $\mathbf P^2$ and on Hirzebruch surfaces.
In particular, this means that there do not exist canonically embedded double structures on smooth
surfaces of minimal degree. This fact is interesting since there do exist double structures of other
kinds on a surface of minimal degree (see ~\cite{GPcarpets} where the existence of double
structures on rational normal scrolls with the invariants of a smooth $K3$ surface is shown).

\medskip

Finite covers of rational surfaces have interesting implications for the geography and the moduli of surfaces of general type (see e.g.~\cite{Cat.moduli} 
 or~\cite{Hor2p_g-4}). In Section~\ref{moduli.section} we chart the region in the geography 
covered by the double covers $X$ of Theorem~\ref{coversofP2andruled}. This is done in Propositions~\ref{cancover.P2.invariants} and~\ref{cancover.ruled.invariants} and in Remark ~\ref{cancover.ruled.geography}. The Chern quotient $\frac{c_1^2}{c_2}$ of our surfaces approaches $\frac{1}{2}$;  
 on the other hand, our region is not contained but goes well inside the region above Castelnuovo's line.  In the second part of Section~\ref{moduli.section} we explore the moduli space of $[X]$. First, in Proposition~\ref{moduli.dim}, we compute  the dimension of the irreducible moduli component of $[X]$, which is  $2d^2+15d+19$ if $\, Y$ is $\mathbf P^2$ embedded by $|\mathcal O_{\mathbf P^2}(d)|$ and $(2a+5)(2b-ae+5)-7$ if $\, Y$ is $\mathbf F_e$ embedded by an arbitrary very ample linear series
$|\mathcal O_Y(aC_0+bf)|$.
Finally we go a bit further in studying the complexity of some
of these
moduli spaces.
We
find examples (see Examples~\ref{example39} and~\ref{example45}) of moduli spaces having two kind of components: components parametrizing surfaces which can be canonically embedded and components parametrizing surfaces whose canonical map is a degree $2$ morphism.

\section{Canonical double covers of 
Hirzebruch surfaces and $\mathbf P^2$}\label{non.existence.onminimal.section}

In this section we study how the canonical morphism of canonical double covers of $\mathbf P^2$ and of Hirzebruch surfaces deforms.
We introduce the notation that we will use in this section and in the remaining of the article:

\begin{noname}\label{setup}
{\bf  Set--up and notation:} {\rm Throughout this article 
we will use the following notation and will follow these conventions:
\begin{enumerate}
\item  We will work over an algebraically closed field $\mathbf k$ of characteristic $0$.
\item  $X$ and $Y$ will be smooth, irreducible projective varieties.
\item  $i$ will denote a projective embedding $i: Y \hookrightarrow \mathbf P^N$ induced by a complete linear series on $Y$. In this case, $\mathcal I$  will denote the ideal sheaf of $i(Y)$ in $\mathbf P^N$. Likewise, we will often abridge $i^*\mathcal O_{\mathbf P^N}(1)$ as $\mathcal O_Y(1)$.
\item $\pi$ will denote a finite morphism $\pi: X \longrightarrow Y$ of degree $n \geq 2$; in this case, $\mathcal E$ will denote the trace--zero module of $\pi$ ($\mathcal E$ is a vector bundle on $Y$ of rank $n-1$).
\item $\varphi$ will denote a projective morphism $\varphi: X \longrightarrow \mathbf P^N$ such that $\varphi= i \circ \pi$.
\item $\mathbf F_e$ will denote the Hirzebruch surface whose minimal section $C_0$ has self--intersection $C_0^2=-e$. We will denote by $f$ the fiber of $\mathbf F_e$ onto $\mathbf P^1$.
\end{enumerate}}
\end{noname}

\noindent
For the reader's convenience, we include here without a proof two results from~\cite{MP} (Theorem 2.6 and Lemma 3.9) that we will use throughout the article. To state the first one we need also Proposition 3.7 from~\cite{Gon}.
We assume the notation just stated but point out that Proposition~\ref{morphism.miguel} and Theorem~\ref{Psi2=0} hold also for embeddings $i$ induced by linear series non necessarily complete.

\begin{proposition}\label{morphism.miguel} 
There exists a homomorphism
\begin{equation*}
 H^0(\mathcal N_\varphi) \overset{\Psi}\longrightarrow \mathrm{Hom}(\pi^*(\mathcal I/\mathcal I^2), \mathcal O_X),
\end{equation*}
that appears when taking cohomology on the commutative diagram~\cite[(3.3.2)]{Gon}. Since
\begin{equation*}
\mathrm{Hom}(\pi^*(\mathcal I/\mathcal I^2), \mathcal O_X)=\mathrm{Hom}(\mathcal I/\mathcal I^2, \pi_*\mathcal O_X)=\mathrm{Hom}(\mathcal I/\mathcal I^2, \mathcal O_Y) \oplus \mathrm{Hom}(\mathcal I/\mathcal I^2, \mathcal E)
\end{equation*}
the homomorphism $\Psi$ has two components
\begin{eqnarray*}
H^0(\mathcal N_\varphi) & \overset{\Psi_1}  \longrightarrow & \mathrm{Hom}(\mathcal I/\mathcal I^2, \mathcal O_Y) \cr
H^0(\mathcal N_\varphi) & \overset{\Psi_2}  \longrightarrow & \mathrm{Hom}(\mathcal I/\mathcal I^2, \mathcal E).
\end{eqnarray*}
\end{proposition}

\begin{theorem}\label{Psi2=0}
With the notation of~\ref{setup}
and of Proposition~\ref{morphism.miguel}, let $X$ be a smooth variety of general type of dimension $m \geq 2$ with ample and base--point--free canonical bundle, let $\varphi$ be its canonical morphism and assume that the degree of $\pi$ is $n=2$.
Assume furthermore that
\begin{enumerate}
\item $h^1(\mathcal O_Y)=h^{m-1}(\mathcal O_Y)=0$ (in particular, $Y$ is regular);
\item $h^1(\mathcal O_Y(1))=h^{m-1}(\mathcal O_Y(1))=0$;
\item $h^0(\omega_Y(-1))=0$;
\item $h^1(\omega_Y^{-2}(2))=0$;
\item the variety $Y$ is unobstructed in $\mathbf P^N$ (i.e., the base of the universal deformation space of $Y$ in $\mathbf P^N$ is smooth); and
\item $\Psi_2 = 0$.
\end{enumerate}
Then
\begin{enumerate}
\item[(a)] $X$ and $\varphi$ are unobstructed (for a definition of unobstructedness, see e.g.~\cite{Ser}), and
\item[(b)] any deformation of $\varphi$ is a (finite) canonical morphism of degree $2$. Thus the canonical map of a variety  corresponding to a general point of the component of $X$ in its moduli space is a finite morphism of degree $2$.
\end{enumerate}
 \end{theorem}

\begin{lemma}\label{isom.Hom.Ext} 
Let $S$ be a (smooth) surface, embedded in $\mathbf P^N$, let $\mathcal J$ be the ideal sheaf of $S$ in $\mathbf P^N$  and consider the connecting homomorphism
\begin{equation*}
\mathrm{Hom}(\mathcal J/\mathcal J^2,\omega_{S}(-1)) \overset{\delta} \longrightarrow \mathrm{Ext}^1(\Omega_{S},\omega_{S}(-1)).
\end{equation*}
\begin{enumerate}
\item If $p_g(S)=0$ and $h^1(\mathcal O_{S}(1))=0$, then $\delta$ is injective;
\item if $q(S)=0$ and $S$ is embedded by a complete linear series,
then $\delta$ is surjective.
\end{enumerate}
\end{lemma}

\noindent
As we have seen in~\cite{Gon}, \cite{Compositio} or~\cite{MP}, the study of the deformations of $\varphi$ is linked to the space Hom$(\mathcal I/\mathcal I^2, \mathcal E)$. Since, as we will see later, if $\varphi$ is a canonical double cover (and $h^0(\omega_Y(-1)=0$), then $\mathcal E=\omega_Y(-1)$, we will study $\mathrm{Hom}(\mathcal I/\mathcal I^2, \mathcal \omega_Y(-1))$ in the next proposition:

\begin{proposition}\label{nonexist.canonical.structures} Let $Y$ be either $\mathbf P^2$ or a Hirzebruch surface and let $i$ be as in \ref{setup}.
Then
 $\mathrm{Hom}(\mathcal I/\mathcal I^2, \mathcal \omega_Y(-1))=0$.
\end{proposition}

\begin{proof}
 First we argue when $Y$ is a Hirzebruch surface. Let $\mathcal O_Y(1)=\mathcal O_Y(aC_0+bf)$.  
Since \linebreak
$\mathcal O_Y(aC_0+bf)$ is very ample, we have
\begin{equation}\label{Hirz.veryample}
 b-ae \geq 1.
\end{equation}
We can apply Lemma~\ref{isom.Hom.Ext};
indeed, $p_g(Y)=0$, $h^1(\mathcal O_Y(1))=0$,
$q(Y)=0$
and $h^2(\mathcal O_Y(1))=0$.
Then, Lemma~\ref{isom.Hom.Ext}
says that it suffices to see that $\mathrm{Ext}^1 (\Omega_Y,  \omega_Y(-1))=0$.
For this, apply~\cite[Proposition II.8.11]{Hart} to the fibration of $Y$ to $\mathbf P^1$ (see also the proof of~\cite[Proposition 1.7]{GPcarpets}) and get the sequence
\begin{equation}\label{fibration}
0 \longrightarrow \mathcal O_Y(-2f) \longrightarrow \Omega_Y \longrightarrow \mathcal O_Y(-2C_0-ef) \longrightarrow 0.
\end{equation}
Then, applying Hom$(-,\omega_Y(-1))$ we get
\begin{equation*}
 \mathrm{Ext}^1(\mathcal O_Y(-2C_0-ef), \omega_Y(-1)) \longrightarrow \mathrm{Ext}^1(\Omega_Y, \omega_Y(-1)) \longrightarrow  \mathrm{Ext}^1(\mathcal O_Y(-2f), \omega_Y(-1)).
\end{equation*}
Now $\mathrm{Ext}^1(\mathcal O_Y(-2C_0-ef), \omega_Y(-1))
=H^1(\mathcal O_Y((a-2)C_0+(b-e)f))^{\vee}=0$ (by pushing down to $\mathbf P^1$ and because of~\eqref{Hirz.veryample}). Also
$\mathrm{Ext}^1(\mathcal O_Y(-2f), \omega_Y(-1))=H^1(\mathcal O_Y(aC_0+(b-2)f))^{\vee}=0$ (again by pushing down to $\mathbf P^1$ and because of~\eqref{Hirz.veryample}), so $\mathrm{Ext}^1(\Omega_Y, \omega_Y(-1))=0$.
Then, by Lemma~\ref{isom.Hom.Ext},
it follows that Hom$(\mathcal I/\mathcal I^2, \omega_Y(-1))=0$.

\smallskip

\noindent Now we prove the proposition for $Y=\mathbf P^2$ embedded by $\mathcal O_{Y}(1)=\mathcal O_{\mathbf P^2}(d)$. Also in this occasion $p_g(Y)=0$, $h^1(\mathcal O_Y(1))=0$,
$q(Y)=0$
and $h^2(\mathcal O_Y(1))=0$, so
by Lemma~\ref{isom.Hom.Ext}
it suffices to see that $\mathrm{Ext}^1 (\Omega_Y, \omega_Y(-1))=0$.
For this we use the Euler sequence of $\mathbf P^2$
\begin{equation}\label{euler.P2}
0 \longrightarrow \Omega_{\mathbf P^2} \longrightarrow H^0(\mathcal O_{\mathbf P^2}(1)) \otimes  \mathcal O_{\mathbf P^2}(-1) \longrightarrow \mathcal O_{\mathbf P^2} \longrightarrow 0.
\end{equation}
To the sequence~\eqref{euler.P2} we apply the functor  Hom$(-,\omega_Y(-1))$ to obtain the exact sequence
\begin{equation*}
0 \longrightarrow \mathrm{Ext}^1(\Omega_{\mathbf P^2},\omega_Y(-1)) \longrightarrow \mathrm{Ext}^2(\mathcal O_{\mathbf P^2}, \omega_Y(-1))
\longrightarrow  \mathrm{Ext}^2(H^0(\mathcal O_{\mathbf P^2}(1))  \otimes \mathcal O_{\mathbf P^2}(-1), \omega_Y(-1)).
\end{equation*}
Dualizing we get
\begin{equation*}
H^0(\mathcal O_{\mathbf P^2}(d-1)) \otimes H^0(\mathcal O_{\mathbf P^2}(1)) \overset{\alpha} \longrightarrow H^0(\mathcal O_{\mathbf P^2}(d)) \longrightarrow \mathrm{Ext}^1(\Omega_{\mathbf P^2},\omega_Y(-1))^{\vee} \longrightarrow 0.
\end{equation*}
Now the multiplication map $\alpha$ is surjective if $d \geq 1$, so $\mathrm{Ext}^1(\Omega_{\mathbf P^2},\omega_Y(-1))$ vanishes and, by Lemma~\ref{isom.Hom.Ext},
so does Hom$(\mathcal I/\mathcal I^2, \omega_Y(-1))$.
\end{proof}

\noindent
We will use Proposition~\ref{nonexist.canonical.structures} to obtain the main result of this section, Theorem~\ref{coversofP2andruled}. In order to do this we need first the following

\begin{lemma}\label{Yunobs}
Let $Y$ be either $\mathbf P^2$ or a Hirzebruch surface and let $i$ be as in~\eqref{setup}. Then $H^1(\mathcal N_{i(Y)/\mathbf P^N})=0$. In particular, $i(Y)$ is unobstructed in $\mathbf P^N$.
\end{lemma}

\begin{proof}
Recall that $H^1(\mathcal N_{i(Y)/\mathbf P^N})=\mathrm{Ext}^1(\mathcal I/\mathcal I^2,\mathcal O_Y)$, which fits into the exact sequence
\begin{equation}\label{Yunobs.seq}
 \mathrm{Ext}^1(\Omega_{\mathbf P^N}|_Y,\mathcal O_Y) \longrightarrow \mathrm{Ext}^1(\mathcal I/\mathcal I^2,\mathcal O_Y) \longrightarrow \mathrm{Ext}^2(\Omega_Y,\mathcal O_Y).
\end{equation}
We want to see that both $\mathrm{Ext}^1(\Omega_{\mathbf P^N}|_Y,\mathcal O_Y)$ and $\mathrm{Ext}^2(\Omega_Y,\mathcal O_Y)$ vanish. To handle the vanishing of $\mathrm{Ext}^1(\Omega_{\mathbf P^N}|_Y,\mathcal O_Y)$, consider the sequence
\begin{equation}\label{Yunobs.seq2}
\mathrm{Ext}^1(\mathcal O_Y^{N+1}(-1),\mathcal O_Y) \longrightarrow \mathrm{Ext}^1(\Omega_{\mathbf P^N}|_Y,\mathcal O_Y) \longrightarrow \mathrm{Ext}^2(\mathcal O_Y,\mathcal O_Y).
\end{equation}
It is clear that $\mathrm{Ext}^1(\mathcal O_Y^{N+1}(-1),\mathcal O_Y)$ and $\mathrm{Ext}^2(\mathcal O_Y,\mathcal O_Y)$ both vanish because $h^1(\mathcal O_{Y}(1))=0$ and $p_g(Y)=0$.
\smallskip
\noindent To prove the vanishing of $\mathrm{Ext}^2(\Omega_Y,\mathcal O_Y)$ we argue for $\mathbf P^2$ and for Hirzebruch surfaces separately. If $Y=\mathbf P^2$, after applying Hom$(-,\mathcal O_Y)$ to~\eqref{euler.P2} we see that $\textrm{Ext}^2(\Omega_{\mathbf P^2},\mathcal O_{\mathbf P^2})$ fits into the exact sequence
\begin{equation*}
\mathrm{Ext}^2(\mathcal O_{\mathbf P^2}(-1), \mathcal O_{\mathbf P^2})^{\oplus 3} \longrightarrow \mathrm{Ext}^2(\Omega_{\mathbf P^2},\mathcal O_{\mathbf P^2}) \longrightarrow \mathrm{Ext}^3(\mathcal O_{\mathbf P^2}, \mathcal O_{\mathbf P^2}),
\end{equation*}
so $\textrm{Ext}^2(\Omega_Y,\mathcal O_Y)=0$. Now, if $Y$ is a Hirzebruch surface we apply Hom$(-,\mathcal O_Y)$ to
\eqref{fibration} and get
\begin{equation*}
 \mathrm{Ext}^2(\mathcal O_Y(-2C_0-ef), \mathcal O_Y) \longrightarrow \mathrm{Ext}^2(\Omega_Y,\mathcal O_Y) \longrightarrow  \mathrm{Ext}^2(\mathcal O_Y(-2f), \mathcal O_Y).
\end{equation*}
Then $\mathrm{Ext}^2(\Omega_Y,\mathcal O_Y)=0$ because $H^2(\mathcal O_Y(2C_0+ef))=H^2(\mathcal O_Y(2f))=0$. Thus $H^1(\mathcal N_{i(Y)/\mathbf P^N})=0$, and by~\cite[Corollary 3.2.7]{Ser}, $i(Y)$ is unobstructed in $\mathbf P^N$. 
\end{proof}

\begin{theorem}\label{coversofP2andruled}
Let $X$ be a smooth surface of general type with ample and base--point-free canonical line bundle and let $\varphi$ be its canonical morphism. Assume that the degree of $\pi$ is $n=2$
and that $Y$ is either $\mathbf P^2$ or a Hirzebruch surface and let $i$ be induced by the complete linear series of a very ample divisor $D$ on $Y$ (if $\, Y$ is a Hirzebruch surface, then $D=aC_0+bf$ with $b-ae \geq 1$ and $a \geq 1$).
Moreover, if $\, Y$ is the Hirzebruch surface assume that
$\omega_Y^{-2}(2)$ is base--point--free. 
Then
\begin{enumerate}
\item[(a)] $X$ and $\varphi$ are unobstructed, and
\item[(b)] any deformation of $\varphi$ is a canonical $2:1$ morphism. Thus the canonical map of a variety  corresponding to a general point of the component of $X$ in its moduli space is a finite morphism of degree $2$.
\end{enumerate}
\end{theorem}

\begin{proof}
We will use Theorem~\ref{Psi2=0}, so we check now that its hypotheses are satisfied.
Condition (1) holds because the surface $Y$ is regular. Condition (2) of Theorem~\ref{Psi2=0} is
$h^1(\mathcal O_Y(1))=0$ and this has been already observed in the proof of Lemma~\ref{Yunobs}.
Condition (3) follows from the fact that $p_g(Y)=0$.
To check Condition (4) observe that, since $X$ is smooth, so is its branch locus, which is a divisor in $|\omega_Y^{-2}(2)|$ since $\varphi$ is
the canonical morphism of $X$. If $Y$ is a Hirzebruch surface,
then $\omega_Y^{-2}(2)$ is base--point-free by assumption. Then  all the twists of the push--down of $\omega_Y^{-2}(2)$ to $\mathbf P^1$ are non negative, so Condition (4) of Theorem~\ref{Psi2=0} holds in this case. If $Y=\mathbf P^2$, then Condition (4) of Theorem~\ref{Psi2=0} follows trivially from the vanishing of intermediate cohomology on $\mathbf P^2$. Condition (5) of Theorem~\ref{Psi2=0} follows from Lemma~\ref{Yunobs}. Finally recall that, since $\varphi$ is the canonical morphism of $X$ and Condition (3) of Theorem~\ref{Psi2=0} holds, the trace zero module of $\pi$ is $\omega_Y(-1)$. This follows from relative duality, having in account that on $Y$ numerical equivalence is the same as linear equivalence (see also Lemma~\ref{canon.cover.split} for a more general statement). Then Condition (6) of Theorem~\ref{Psi2=0} follows from Proposition~\ref{nonexist.canonical.structures}.
\end{proof}

\noindent The assumption made in the statement of Theorem~\ref{coversofP2andruled} asking  $\omega_Y^{-2}(2)$ to be base--point--free if $Y$ is a Hirzebruch surface is only needed if $e \geq 6$ and even, in which case it is not a very strong assumption:

\begin{remark}\label{freeness.remark} Let $Y=\mathbf F_e$ and assume that $X, Y$ and $\varphi$ are as in Theorem~\ref{coversofP2andruled} except that no base--point-freeness assumption is made on $\omega_Y^{-2}(2)$. Then $\omega_Y^{-2}(2)$ is base--point-free unless $e$ is even, $e \geq 6$ and $b-ae=\frac{1}{2}e-2$. In fact, under our hypotheses we have  $b-ae \geq \frac{1}{2}e-2$.
Therefore, when $ e$ is even and $e \geq 6$, the extra 
condition we have to impose on a Hirzebruch surface so that $\omega_Y^{-2}(2)$ be base--point--free is
$b-ae \neq \frac{1}{2}e-2$, or equivalently $b-ae > \frac{1}{2}e-2$.
\end{remark}

\begin{proof}
Since $X$ is smooth, $|\omega_Y^{-2}(2)|$ should have a smooth member. Recall that $\omega_Y^{-2}(2)=\mathcal O_Y((2a+4)C_0+(2b+2e+4)f)$. Then $2b-2ae -e + 4=(2b+2e+4)-(2a+3)e \geq 0$, otherwise $2C_0$ would be in the fixed part of any divisor of $|\omega_Y^{-2}(2)|$. Now, if $\omega_Y^{-2}(2)$ is not base--point-free, then $2b-2ae -e + 4 < e$. On the other hand, if $0 \leq 2b-2ae -e + 4 < e$, the divisor $(2a+4)C_0+(2b+2e+4)f$ is the sum of a fixed part $C_0$ and a base--point--free part $(2a+3)C_0+(2b+2e+4)f$, so $|\omega_Y^{-2}(2)|$ has smooth members if and only if $2b-2ae -e + 4=((2a+3)C_0+(2b+2e+4)f)\cdot C_0=0$. This only happens if $e$ is even and $b-ae=\frac{1}{2}e-2$. Since under our hypotheses $aC_0+bf$ is very ample, $b-ae \geq 1$, so if  $\omega_Y^{-2}(2)$ is not base--point--free,  we have in addition $e \geq 6$.
\end{proof}

\section{Non existence of canonical ropes}\label{noropes}

The deformation of morphisms is related to the appearance or non appearance of multiple structures embedded in projective space. In this section we apply Proposition~\ref{nonexist.canonical.structures} to prove the non existence of ``canonical, embedded ropes'' on $i(Y)$, i.e., multiple structures on $i(Y)$ with conormal bundle $\mathcal E$, when $Y,i,\varphi$ and $\mathcal E$ are as in~\ref{setup} and, in addition,  $\varphi$ is  a canonical morphism and $Y$ is either a Hirzebruch surface or $\mathbf P^2$. For that we need the following lemma on the structure of $\mathcal E$:

\begin{lemma}\label{canon.cover.split}
Let $X$ be a variety of general type with ample and base--point-free canonical divisor and let $\varphi$ be the canonical morphism of $X$. Assume
\begin{equation}\label{nulhyp}
H^0(\omega_Y(-1))=0.
\end{equation}
Then the trace zero module of $\pi$ splits as
\begin{equation*}\label{split}
\SE= \SE' \oplus \omega_Y(-1),
\end{equation*}
where $\SE'$ is a rank $(n-2)$ locally free sheaf.
\end{lemma}

\begin{proof}
Our hypothesis says that $\omega_X = \pi^* \SO_Y(1)$ and
\begin{equation}\label{series}
H^0(\omega_X)=H^0(\SO_Y(1)).
\end{equation}
Moreover, since $\pi_* \SO_X= \SO_Y \oplus \SE$, we have
\begin{equation}\label{pushomega}
\pi_* \omega_X = \SO_Y(1) \oplus (\SE \otimes \SO_Y(1)).
\end{equation}
So $H^0(\omega_X)=H^0(\SO_Y(1)) \oplus H^0(\SE \otimes \SO_Y(1))$. Hence \eqref{series} is equivalent to
\begin{equation}\label{vanishing}
H^0(\SE \otimes \SO_Y(1))=0.
\end{equation}
On the other hand, relative duality for $\pi$ yields $\pi_* \omega_X = (\pi_* \SO_X)^\vee \otimes \omega_Y$ or equivalently
\begin{equation}\label{relative}
\pi_* \omega_X = \omega_Y \oplus (\SE^\vee \otimes \omega_Y).
\end{equation}

\noindent
From \eqref{pushomega} and \eqref{relative}  we get a sequence
\begin{equation*} 0 \longrightarrow \SO_Y(1) \overset{\iota} \longrightarrow \omega_Y \oplus (\SE^{\vee} \otimes \omega_Y) \longrightarrow  \SE \otimes \SO_Y(1)\longrightarrow 0.
\end{equation*}
Then~\eqref{nulhyp} implies that $\iota$ factors through $\SE^{\vee} \otimes \omega_Y$ so we get a diagram
\begin{equation*}\label{diag2}
\xymatrix@C-5pt@R-7pt{
         &   & 0 \ar[d]           & 0 \ar[d]               &   \\
0 \ar[r] & \SO_Y(1) \ar@{=}[d]\ar[r] & \SE^{\vee} \otimes \omega_Y \ar[d] \ar[r] & \SE{'}(1) \ar[d] \ar[r]& 0 \\
0 \ar[r] & \SO_Y(1)  \ar[r]^-\iota & \omega_Y \oplus (\SE^{\vee} \otimes \omega_Y) \ar[d]\ar[r] &  \SE \otimes \SO_Y(1)\ar[d]\ar[r] &0\\
   &       & \omega_Y  \ar[d] \ar@{=}[r] &  \omega_Y  \ar[d]&  \\
   &       &      0     &            0,  &}
\end{equation*}
where $\SE'$ is a rank $(n-2)$ locally free sheaf on $Y$. Besides, both middle exact sequences split, and so the top horizontal and the right--hand vertical exact sequences also split.
Now, from the splitting of the right hand side vertical sequence we get
\begin{equation}\label{vert-split}
\SE = \SE' \oplus  \omega_Y(-1).
\end{equation}
\end{proof}

\noindent
The next corollary shows that there are no canonical double structures on either $\mathbf P^2$ or a Hirzebruch surface. This result 
contrasts with the results of~\cite{GPcarpets} where the existence of double structures on smooth rational normal scrolls having the same invariants of smooth $K3$ surfaces is shown.

\begin{corollary}\label{nonexist.canonical.structures.cor} Let $Y$ as in~\ref{setup} and either $\mathbf P^2$ or a Hirzebruch surface and let $X, \varphi$ and $i$ be also as in~\ref{setup}. Assume furthermore that $X$ is a  surface of general type and that $\varphi$ is its
canonical morphism.
Then
\begin{enumerate}
\item $\mathrm{Hom}(\mathcal I/\mathcal I^2, \mathcal E)$ does not contain surjective homomorphisms;
\item there are no multiple structures inside $\mathbf P^N$, supported on $i(Y)$ with conormal bundle $\mathcal E$;  and
\item there are no double structures inside $\mathbf P^N$, supported on $i(Y)$ whose conormal bundle is a subsheaf of $\mathcal E$.
\end{enumerate}
\end{corollary}

\begin{proof}
Since $p_g(Y)=0$, then $h^0(\omega_Y(-1))=0$. Thus Lemma~\ref{canon.cover.split} implies
$\mathcal E= \mathcal E' \oplus \omega_Y(-1)$ so
\begin{equation}\label{Hom.split}
\mathrm{Hom}(\mathcal I/\mathcal I^2,\mathcal E)=\mathrm{Hom}(\mathcal I/\mathcal I^2,\mathcal E') \oplus \mathrm{Hom}(\mathcal I/\mathcal I^2,\omega_Y(-1)).
\end{equation}
The last summand in~\eqref{Hom.split} is $0$ by Proposition~\ref{nonexist.canonical.structures}, so we get (1).
Parts (2) and (3) follow from~\cite[Proposition 2.1]{Gon}.
\end{proof}

\section{Consequences for geography and moduli}\label{moduli.section}

In this section we compute the invariants of the surfaces of general type we have constructed in Section~\ref{non.existence.onminimal.section}, thus finding the region of the geography of surfaces of general type they reside in. In addition, we compute the dimension of the moduli components parametrizing our surfaces. Finally we show two examples of moduli spaces having components of different nature: components parametrizing canonically embedded surfaces and components parametrizing surfaces whose canonical map is a finite morphism of degree $2$.

\begin{proposition}\label{cancover.P2.invariants}
Let $Y=\mathbf P^2$ embedded by $|\mathcal O_{\mathbf P^2}(d)|$ and let $X$ be a canonical double cover of $Y$  as the ones appearing in Theorem~\ref{coversofP2andruled}.  Then the surface $X$ has the following invariants:
\begin{eqnarray}\label{invariant.formulae.P2}
 p_g&=&\frac{1}{2}d^2+\frac{3}{2}d+1 \cr
 \cr
q&=&0 \cr
\cr
\chi&=&\frac{1}{2}d^2+\frac{3}{2}d+2 \cr
\cr
c_1^2&=&2d^2 \cr
 \cr
\frac{c_1^2}{c_2}&=&\frac{d^2}{2d^2+9d+12}.
\end{eqnarray}
\end{proposition}

\begin{proof}
In our situation $p_g=h^0(\mathcal O_{\mathbf P^2}(d))$. Since $h^1(\mathcal O_X)=h^1(\mathcal O_Y) + h^1(\omega_Y(-1))$ and  $q(Y)=h^1(\mathcal O_Y(1))=0$, $q(X)=0$.
The values of $c_1^2$ are obvious, since $\varphi$ has degree $2$ onto $Y$. Finally, the values of $\frac{c_1^2}{c_2}$ follow from Noether's formula.
\end{proof}

\begin{proposition}\label{cancover.ruled.invariants}
Let $Y$ be  a Hirzebruch surface $Y=\mathbf F_e$ embedded by a very ample linear system $|aC_0+bf|$ and let $X$ be a
canonical double cover of $\, Y$  as the ones appearing in Theorem~\ref{coversofP2andruled}.
Then $X$ has the following invariants:
\begin{eqnarray}\label{invariant.formulae}
 p_g&=&(a+1)(b+1-\frac{ae}{2}) \cr
q&=&0 \cr
\chi&=& (a+1)(b+1-\frac{ae}{2})+1\cr
c_1^2&=&2a(2b-ae) \cr
 \cr
\frac{c_1^2}{c_2}&=&\frac{2ab-a^2e}{4ab-2a^2e+6a-3ae+6b+12}.
\end{eqnarray}
\end{proposition}

\begin{proof}
Since $\varphi$ is a canonical cover, $p_g$ is the dimension of $H^0(\mathcal O_Y(aC_0+bf))$, which is well--known.  By the construction of $X$,
$h^1(\mathcal O_Y)=h^1(\mathcal O_Y) + h^1(\omega_Y(-1))$. Again we know that $q(Y)$ and $h^1(\mathcal O_Y(1))$ are both $0$, so $q(X)=0$ and then the values stated for $\chi$ are obvious. The values of $c_1^2$ follow also from the construction and properties of $X$ and $\varphi$; indeed, $\omega_X=\varphi^*\mathcal O_Y(1)$ and $\varphi$ has degree $2$ onto its image $Y$. Finally, the values of $\frac{c_1^2}{c_2}$ follow from Noether's formula.
\end{proof}

\begin{remark}\label{Chern.quotient} For any of the surfaces $X$ in Propositions~\ref{cancover.P2.invariants} and~\ref{cancover.ruled.invariants}, $\frac{c_1^2}{c_2} < \frac{1}{2}$ but there exist surfaces $X$ as in Proposition~\ref{cancover.P2.invariants} and surfaces $X$ as in Proposition~\ref{cancover.ruled.invariants} for which $\frac{c_1^2}{c_2}$ is arbitrarily close to $\frac{1}{2}$.
\end{remark}

\begin{proof} If $X$ is as in Proposition~\ref{cancover.P2.invariants} and~\ref{cancover.ruled.invariants}, the claim is clear.
If $X$ is as in Proposition~\ref{cancover.ruled.invariants}, note that
\begin{equation*}
 \frac{c_1^2}{c_2}=\frac{2ab-a^2e}{4ab-2a^2e+6a-3ae+6b+12}=\frac{a(2b-ae)}{2a(2b-ae)+6a+3b+3(b-ae)+12}.
\end{equation*}
The very ampleness of $\mathcal O_Y(aC_0+bf)$ implies $b-ae \geq 1$, so $6a+3b+3(b-ae)+12 > 0$ and the claim is also clear in this case.
\end{proof}

\begin{remark}\label{cancover.ruled.geography}
{\rm We now present the information given in Proposition~\ref{cancover.ruled.invariants} more graphically, by displaying on a plane the pairs $(x,y)=(\chi,c_1^2)$ of the covers of ruled surfaces  appearing in Theorem~\ref{coversofP2andruled}. If we fix an integer  $a \geq 1$, then the points $(x,y)$ corresponding to the  covers of ruled surfaces appearing in  Theorem~\ref{coversofP2andruled} are points (with integer coordinates) lying on the line $l_{a}$ passing through the point $(a+2,0)$ with slope $\frac{4a}{a+1}$, i.e., the line of equation
\begin{equation}\label{graphic}
 y=\frac{4a}{a+1}(x-a-2).
\end{equation}

\noindent
More precisely, for each $a \geq 1$, the invariants form an unbounded set consisting of all the integer points on  the semiline of $l_{a}$ including and up and to the right of the point $(2a+3,4a)$.
Note that each  two distinct lines $l_a$ and $l_{a'}$ described above meet at the point with $x=aa'+a+a'+2$. Note also that $l_1$ is obviously the Noether's line $y-2x+6=0$; the reason for this is that if $a=1$, $Y$ is embedded as a surface of minimal degree. Note also that the limiting points of the semilines described above lie on Noether's line $y-2x+6=0$ (as should be, since in this case the limiting point $(2a+3,4a)$ is obtained when considering canonical double covers of $\mathbf F_0$ embedded by $|aC_0+f|$, which are surfaces of minimal degree).

\smallskip
\noindent
Note that as $a$ goes to infinity, the slopes of the lines $l_a$ approaches $4$. This means that Chern ratio $\frac{c_1^2}{c_2}$ approaches $\frac{1}{2}$. Note also that, if $a \geq 4$, when $x$ is sufficiently large, the line $l_a$ goes into the region $y \geq 3x-10$, bounded by the Castelnuovo line $y=3x-10$.
 }\end{remark}

\smallskip

\noindent We illustrate Remark~\ref{cancover.ruled.geography} with Figures 1 and 2. In Figure 1 low values of $(\chi,c_1^2)$ appear while Figure 2 zooms out in order to display larger values. Recall that we denote $\chi$ by $x$ and $c_1^2$ by $y$.

\smallskip

\begin{figure}[!hb]
\begin{center}
\includegraphics[width=0.8\textwidth, 
height=0.6\textheight
]{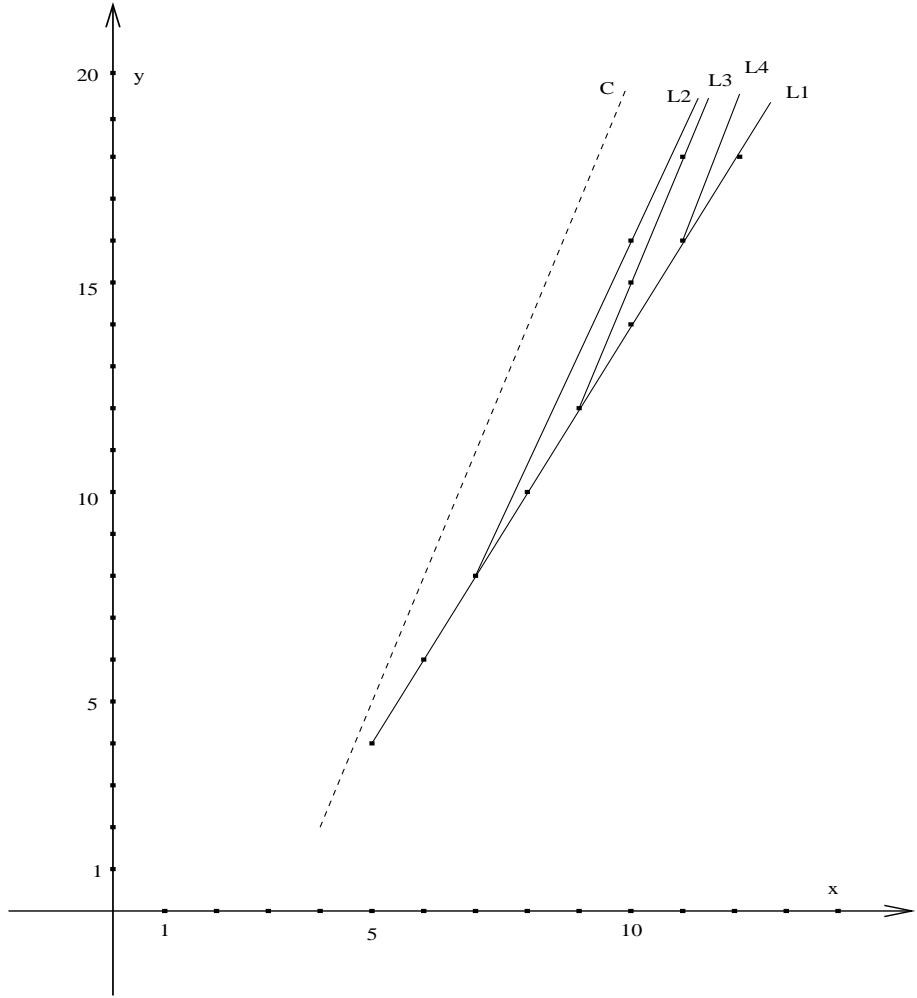}
\end{center}
\caption{Solid lines, from less steep to more steep, are $l_1$ (which is also Noether's line) to $l_4$. The dashed line is Castelnuovo's line. The points marked  are the integer points lying on $l_1$ with first coordinate $x \geq 5$ and  the integer points lying on $l_2, l_3$ and $l_4$ and above $l_1$.}
\end{figure}

\newpage

\begin{figure}[!h]
\begin{center}
\includegraphics[width=0.8\textwidth, height=1.1
\textwidth]{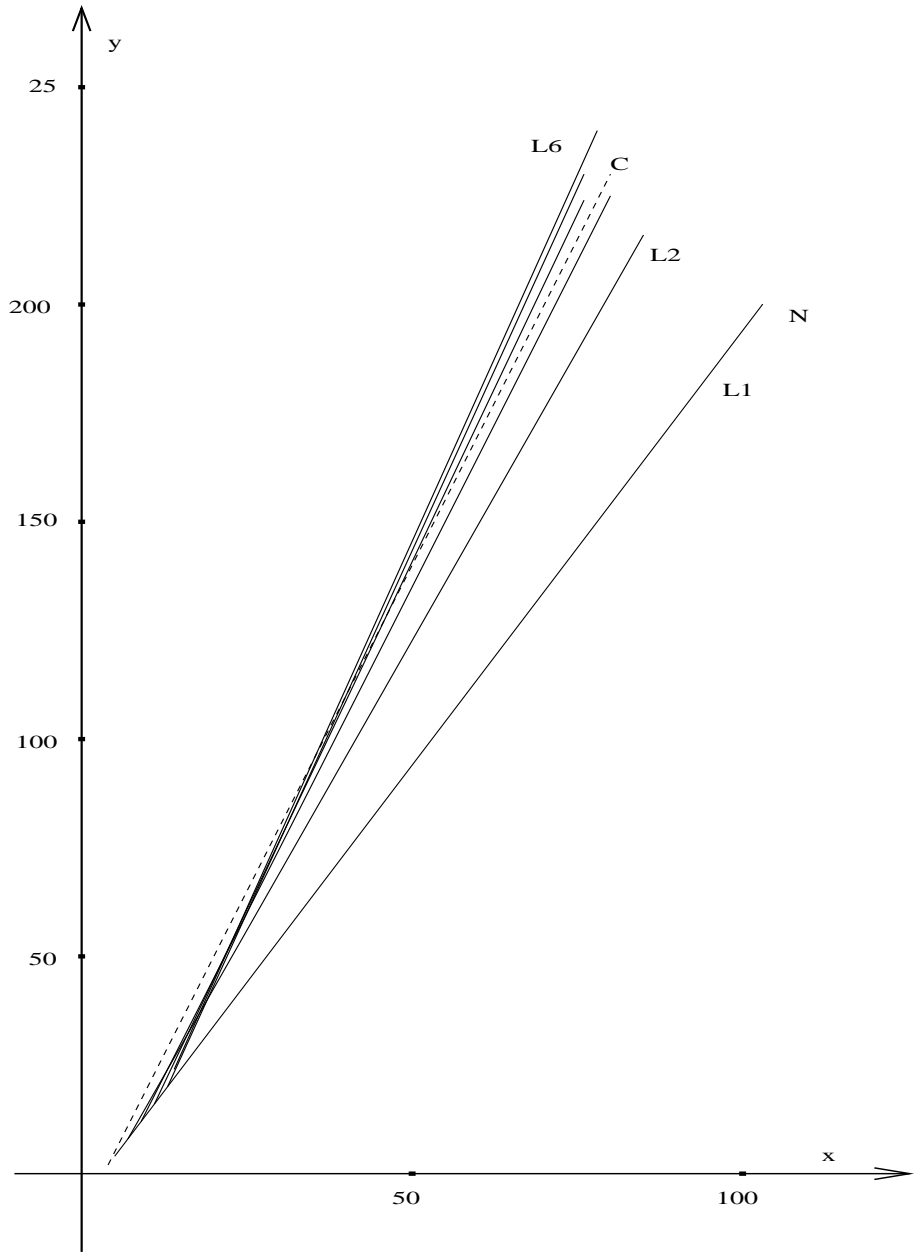}
\end{center}
\caption{Solid lines, from less steep to more steep, are $l_1$ (which is also Noether's line) to $l_6$; the dashed line is Castelnuovo's line.}
\end{figure}

\noindent 
Next we remark that the surfaces constructed in Theorem~\ref{coversofP2andruled} are not only regular, but also simply connected:

\begin{remark}
The surfaces of general type $X$ that we have constructed in Theorem~\ref{coversofP2andruled} are simply connected.
\end{remark}

\begin{proof} Let $Y$ be either a minimal rational surface or $F_1$. The fundamental group of $Y$ is well known to be $0$. The morphism $\pi$ is a double cover of $Y$ branched along a divisor in $|\omega_Y^{-2}(2)|$. If $Y$ is $\mathbf P^2$ this branch divisor is obviously base--point-free
and if $\, Y$ is a Hirzebruch surface, the branch divisor is also base--point--free by 
the hypothesis of Theorem~\ref{coversofP2andruled}. Then the fundamental group of $X$ is 
the same as the fundamental group of $Y$ by~\cite[Corollary 2.7]{Nori} 
(note that the ampleness hypothesis required there can be relaxed to big and nefness), so $X$ is simply connected.
Then, consider the families of surfaces associated to the deformations of $X$ given in Theorem~\ref{coversofP2andruled}.
All the smooth fibers in such families are diffeomorphic to each other, hence they are also simply connected.
Thus the surfaces constructed in Theorem~\ref{coversofP2andruled} are simply connected.
\end{proof}

\medskip
\noindent
In the second part of this section we compute the dimension of the components of the moduli parametrizing the surfaces of general type appearing in Theorem~\ref{coversofP2andruled}:

\begin{proposition}\label{moduli.dim}
 Let $X$ be a surface of general type as in Theorem~\ref{coversofP2andruled}. Then there is only one irreducible component of the moduli containing $[X]$ and its dimension is
\begin{enumerate}
\item $\mu=2d^2+15d+19$, if $\, Y=\mathbf P^2$;
\item $\mu=(2a+5)(2b-ae+5)-7$, if $\, Y$ is a Hirzebruch surface.
\end{enumerate}
\end{proposition}

\begin{proof}
Part (a) of Theorem~\ref{coversofP2andruled} implies that 
the base of the formal semiuniversal deformation space of $X$ is smooth, so in particular $[X]$ belongs to a unique irreducible component 
of the moduli.
Note that \cite[Lemma 4.4]{MP} holds also under our hypothesis. Then
\begin{equation}\label{value.of.mu}
\mu= h^0(\mathcal N_\pi) - h^1(\mathcal N_\pi) + h^1(\mathcal T_Y) - h^0(\mathcal T_Y) + \textrm{ dim Ext}^1 (\Omega_Y,\omega_Y(-1)).
\end{equation}
Thus, we will compute now the dimensions of the cohomology groups that appear in~\eqref{value.of.mu}. First recall that in the proof of Proposition~\ref{nonexist.canonical.structures} we obtained $\textrm{Ext}^1 (\Omega_Y,\omega_Y(-1))=0$.

\smallskip
\noindent Now we prove that $h^1(\mathcal N_\pi)=0$. Recall (see~\cite[Lemma 2.5]{MP}) that $h^1(\mathcal N_\pi)=h^1(\mathcal O_B(B))$, where $B$ is the branch divisor of $\pi$. The divisor $B$ is a smooth member of $|\omega_Y^{-2}(2)|$. Recall also that we showed in the proof of Theorem~\ref{coversofP2andruled} that $H^1(\omega_Y^{-2}(2))=0$.  Then the sequence
\begin{equation}
H^1(\mathcal O_Y) \longrightarrow H^1(\mathcal O_Y(B)) \longrightarrow H^1(\mathcal O_B(B)) \longrightarrow H^2(\mathcal O_Y)
\end{equation}
and the fact that $p_g(Y)=0$ imply the vanishing of $H^1(\mathcal N_\pi)$.

\smallskip
\noindent Next we compute the number $h^0(\mathcal T_Y)-h^1(\mathcal T_Y)$. If $Y=\mathbf P^2$, since $h^2(\mathcal T_Y)=0$, then
\begin{equation}\label{Xi.TP2}
 h^0(\mathcal T_Y)-h^1(\mathcal T_Y)=\chi(\mathcal T_Y)=8.
\end{equation}
Now if $\, Y$ is a Hirzebruch surface, dualizing and taking global sections on~\eqref{fibration} yields
\begin{equation}\label{Xi.THirz}
 0 \longrightarrow H^0(\mathcal O_Y(2C_0+ef)) \longrightarrow H^0(\mathcal T_Y) \longrightarrow H^0(\mathcal O_Y(2f)) \longrightarrow H^1(\mathcal O_Y(2C_0+ef)) \longrightarrow H^1(\mathcal T_Y)\longrightarrow 0.
\end{equation}
Then $h^0(\mathcal T_Y)-h^1(\mathcal T_Y)=6$.

\smallskip
\noindent Now, to complete the computation we find  $h^0(\mathcal N_\pi)$.
Again by~\cite[Lemma 2.5]{MP} we have $h^0(\mathcal N_\pi)=h^0(\mathcal O_B(B))=h^0(\mathcal O_Y(B))-1$, because $Y$ is regular. If $\, Y=\mathbf P^2$, embedded by $|\mathcal O_{\mathbf P^2}(d)|$,
then $h^0(\mathcal N_\pi)=2d^2+15d+27$ and if $\, Y$ is a Hirzebruch surface $\mathbf F_e$ 
as in Theorem~\ref{coversofP2andruled}, then by Riemann--Roch,  $h^0(\mathcal N_\pi)=(2a+5)(2b-ae+5)-1$ 
(recall that $\omega_Y^{-2}(2)$ is base--point--free and that for a Hirzebruch surface $Y$ this implies 
the vanishing of $H^1(\omega_Y^{-2}(2))$).
Plugging all this in~\eqref{value.of.mu} yields the result.
\end{proof}

\noindent
Some of the
surfaces constructed in Theorem~\ref{coversofP2andruled}  provide examples of moduli spaces with
interesting properties:

\begin{example}\label{example39}
The moduli space $\mathcal M_{(39,0,110)}$ parametrizing surfaces of general type with $p_g=39, q=0$ and $c_1^2=110$ has one component $\mathcal M_1$ whose general point corresponds to a surface $S$ as in~\cite[Lemma 4.10]{MP} and another component $\mathcal M_2$ whose general point corresponds to a surface $X$ as in Theorem~\ref{coversofP2andruled}. In particular a general point of $\mathcal M_1$ corresponds to a surface that can be canonically embedded whereas a general point of $\mathcal M_2$ corresponds to a surface whose canonical map is a degree $2$, finite morphism.
\end{example}

\begin{proof} For instance, the linear system $|4H+10F|$ of $S(1,1,1)$ has smooth members $S$ which are surfaces like those of~\cite[Lemma 4.10]{MP} and with $(p_g(S),c_1^2(S))=(39,110)$ ($H$ is the tautological divisor of $S(1,1,1)$ and $F$ is a fiber over $\mathbf P^1$). On the other hand, canonical double covers $X$ of $\mathbf F_1$ embedded by $|5C_0+8f|$ are surfaces like those in Theorem~\ref{coversofP2andruled} and have $(p_g(X),c_1^2(X))=(39,110)$. Since by Theorem~\ref{coversofP2andruled} the canonical map of $X$ deforms to a finite morphism of degree $2$, $S$ and $X$ belong to different components of $\mathcal M_{(39,0,110)}$.
\end{proof}

\begin{example}\label{example45}
The moduli space $\mathcal M_{(45,0,128)}$ parametrizing surfaces of general type with $p_g=45, q=0$ and $c_1^2=128$ has at least three components $\mathcal M_1$, $\mathcal M_2$ and $\mathcal M_3$. A general point of  $\mathcal M_1$  corresponds to a surface $S$ as in~\cite[Lemma 4.10]{MP}, while general points of  $\mathcal M_2$ and $\mathcal M_3$ correspond to surfaces $X$ as in Theorem~\ref{coversofP2andruled}. In particular a general point of $\mathcal M_1$ corresponds to a surface that can be canonically embedded whereas general points of $\mathcal M_2$ and $\mathcal M_3$ correspond to surfaces whose canonical map is a degree $2$, finite morphism.
\end{example}

\begin{proof} For instance, the linear system $|4H+12F|$ of $S(1,1,1)$ has smooth members $S$ which are surfaces like those of~\cite[Lemma 4.10]{MP} and with $(p_g(S),c_1^2(S))=(45,128)$ (recall that $H$ is the tautological divisor of $S(1,1,1)$ and $F$ is a fiber over $\mathbf P^1$).  On the other hand, canonical double covers of $\mathbf P^2$ embedded by octics are surfaces $X_2$ as in Theorem~\ref{coversofP2andruled} having $(p_g(X_2),c_1^2(X_2))=(45,128)$ (see \eqref{invariant.formulae.P2}). In addition, canonical double covers of $\mathbf F_0$ embedded by $|4C_0+8f|$ are also surfaces $X_3$ as in Theorem~\ref{coversofP2andruled} having $(p_g(X_2),c_1^2(X_2))=(45,128)$ (see \eqref{invariant.formulae}). Now recall that the point in the moduli space corresponding to $X_2$ belongs to only one component of $\mathcal M_{(45,0,128)}$ (see Proposition~\ref{moduli.dim}), which we will call $\mathcal M_2$. Likewise, the point in the moduli space corresponding to $X_3$ belongs to only one component of $\mathcal M_{(45,0,128)}$, which we will call $\mathcal M_3$. Then $\mathcal M_2$ and $\mathcal M_3$ are different for their dimensions are: indeed, applying Proposition~\ref{moduli.dim} we get that the dimension of $\mathcal M_2$ is $267$ whereas the dimension of $\mathcal M_3$ is $266$.
\end{proof}

\begin{remark}
For any integer $m$, $m \geq 4$, let $\Xi_m$ be the set of values $(x',y)$ for which there exist a smooth surface $X$ as in Theorem~\ref{coversofP2andruled} with $(p_g(X),c_1^2(X))=(x',y)$ and a smooth surface $S$ as in~\cite[Lemma 4.10]{MP} with $(p_g(S),c_1^2(S))=(x',y)$.
\begin{enumerate}
\item The set $\Xi_m$ is finite (and possibly empty) for every $m \geq 4$. In particular, $\Xi_4=\{(39,110),$ $(45,128)\}$ and $\Xi_5=\Xi_6=\emptyset$.
\item There are no surfaces $X$ with $Y=\mathbf P^2$ such that $(p_g(X),c_1^2(X)) \in \Xi_4 \cup \Xi_5 \cup \cdots$, except the surfaces $X_2$ appearing in Example~\ref{example45}.
\end{enumerate}
\end{remark}

\begin{proof}
The remark follows from elementary although somehow involved computations, once we take in account~\cite[(3.17.1)]{MP}, the hypothesis of~\cite[Lemma 4.10]{MP}, \eqref{invariant.formulae.P2}, \eqref{invariant.formulae}, \eqref{graphic} and the fact that if
 $S$ is as in~\cite[Lemma 4.10]{MP}, then  $(p_g(S),c_1^2(S))=(x',y)$ satisfies the equation
\begin{equation}\label{yx'm}
 y=6\frac{m-3}{m-2}x'-(m-3)(m+3)
\end{equation}
(see~\cite[(3.17.2)]{MP}).
\end{proof}

\begin{acknowledgement}
{\rm We are very grateful to Edoardo Sernesi  for drawing our attention to our earlier work on deformation of morphisms and for suggesting us to use it to study the canonical maps of surfaces of general type. We thank Madhav Nori for a helpful conversation. 
We also thank Tadashi Ashikaga for bringing to our attention his result~\cite[4.5]{AK} with Kazuhiro Konno.
}
\end{acknowledgement}

\end{document}